\documentclass[11pt,a4paper,reqno]{amsart} 
\usepackage{graphicx} 
\usepackage{amssymb}
\usepackage{latexsym}
\usepackage{amsmath}
\usepackage{mathrsfs}
\usepackage{nccmath}    
\usepackage[T1]{fontenc} 
\usepackage{color}
\usepackage{cite}
\renewcommand\labelenumi{(\roman{enumi})}
\renewcommand\theenumi\labelenumi
\usepackage{calc}

\usepackage{hyperref}
\hypersetup{hidelinks}

\newcommand{\scal}[2]{\langle #1,#2\rangle}

\newcommand{\rr}[1]{\mathbf R^{#1}}

\newcommand{\zz}[1]{\mathbf Z^{#1}}

\newcommand{\nm}[2]{\Vert #1\Vert _{#2}}
\newcommand{\Nm}[2]{\big \Vert #1\big \Vert _{#2}}
\newcommand{\NM}[2]{\left \Vert #1\right \Vert _{#2}}
\newcommand{\nmm}[1]{\Vert #1\Vert }

\newcommand{\op}{\operatorname{Op}}

\newcommand{\sets}[2]{\{ \, #1\, ;\, #2\, \} }
\newcommand{\ep}{\varepsilon}
\newcommand{\fy}{\varphi}

\newcommand{\cdo}{\, \! \cdot \, \! }

\newcommand{\supp}{\operatorname{supp}}
\newcommand{\wpr}{{\text{\footnotesize $\#$}}}

\newcommand{\eabs}[1]{\langle #1\rangle}

\newcommand{\vrum}{\vspace{0.1cm}}

\newcommand{\WL}{W\!\! L}
\newcommand{\Sh}{\operatorname{S\hspace{-0.25mm}h}}
\newcommand{\SG}{\operatorname{S}}

\newcommand{\maclB}{\mathcal B}

\newcommand{\maclH}{\mathcal H}

\newcommand{\maclK}{\mathcal K}
\newcommand{\maclL}{\mathcal L}

\newcommand{\maclU}{\mathcal U}
\newcommand{\maclV}{\mathcal V}

\newcommand{\mascC}{\mathscr C}

\newcommand{\mascF}{\mathscr F}

\newcommand{\mascI}{\mathscr I}

\newcommand{\mascS}{\mathscr S}

\numberwithin{equation}{section} 

\newtheorem{thm}{Theorem}
\numberwithin{thm}{section}

\newcommand{\rubrik}{}
\newtheorem{prop}[thm]{Proposition}
\newtheorem{cor}[thm]{Corollary}
\newtheorem{lemma}[thm]{Lemma}

\theoremstyle{definition}

\newtheorem{defn}[thm]{Definition}

\newtheorem{rem}[thm]{Remark}
\theoremstyle{remark}

\title{Quasi-Banach Schatten-von Neumann
properties in Weyl-Hörmander calculus}

\author{Matteo Bonino}

\address{Dipartimento di Matematica ``G. Peano'', Universit\'a
degli Studi di Torino}

\email{matteo.bonino@unito.it}

\author{Sandro Coriasco}

\address{Dipartimento di Matematica ``G. Peano'', Universit\'a
degli Studi di Torino}

\email{sandro.coriasco@unito.it}

\author{Albin Petersson}

\address{Department of Mathematics,
Linn{\ae}us University, Sweden}

\email{albin.petersson@lnu.se}

\author{Joachim Toft}

\address{Department of Mathematics,
Linn{\ae}us University, Sweden}

\email{joachim.toft@lnu.se}

\begin{document}

\begin{abstract}
We study structural properties of
$\WL_{g,\theta} ^{q,p}$, which are
Wiener-Lebesgue spaces
with respect to a slowly varying metric $g$ and 
with parameters $p,q\in(0,\infty ]$, 
$\theta\in\mathbb{R}$. For $p\in (0,1]$,
we deduce Schatten-$p$ 
properties for pseudo-differential
operators whose symbols, together with
their derivatives, obey suitable
$\WL_{g,\theta} ^{q,p}$-boundedness conditions.
Especially, we perform such investigations for the
Weyl-H{\"o}rmander calculus.
Finally, we apply our results to
global-type SG and Shubin
pseudo-differential operators.
\end{abstract}



\maketitle

\section{Introduction}\label{sec0}


The theory of pseudo-differential
operators naturally arises in e.{\,}g.
partial differential equations, statistics,
quantum mechanics, and signal processing.
A pseudo-differential calculus is a rule which
associates a suitable function $a(x,\xi )$,
defined on the phase space
$W=V\times V'\asymp \rr {2d}$,
to a linear operator $\op (a)$. (See
\cite{Ho3} or Section \ref{sec1}
for notations.) The function
$a(x,\xi )$ is called the symbol of $\op (a)$.
The partial differential operators are obtained
by choosing the symbols to be polynomials in
the momentum variable $\xi \in V'$.
Hence, pseudo-differential operators are a
generalization of the concept of differential
operators.

\par

The Weyl quantization $a\mapsto \op ^w(a)$
is unique because it
is the only pseudo-differential calculus
which is invariant under affine symplectic
transformations. This property is fundamental
in quantum mechanics, making the Weyl
quantization of special interest in several
fields. This symplectic structure also facilitates 
calculations which are otherwise more cumbersome. Therefore,
the Weyl calculus naturally lends itself to deeper analysis. 

\par

An important question in the theory pseudo-differential 
operators is
to find suitable conditions on the symbol classes
in order to guarantee $L^2$-continuity and compactness
properties of the corresponding operators. More
detailed studies on compactness are then possible
in the framework of Schatten-von Neumann classes,
a family $\{\mascI_p\}_{p\in (0,\infty ]}$ of
operator spaces characterized by the decay
properties of their singular values.

\par


\par




In the paper, we find sufficient conditions on
symbols in the H{\"o}rmander class $S(m,g)$
in order for corresponding
pseudo-differential operators to be
Schatten operators of degree $0<p\leq 1$ on $L^2$. 


\par

In the case that $1\le p \le \infty$, 
investigations related to ours can be found in 
\cite{Tof6,BuNi,BuTo0}.
%
%
%
%
It is then assumed that the weight function $m$ fulfills
different types of $L^p$ boundedness conditions. More
precisely, suppose that $g$ is strongly feasible on $W$,
$p\in [1,\infty ]$ and $m$ is $g$-continuous
and $(\sigma ,g)$-temperate.
In \cite{Tof6} it is then proved that
\begin{align}
\label{Eq:IntrSchattResult1}
m\in L^p
\quad &\iff \quad
\op ^w(a) \in \mascI _p,
\quad \text{when}\quad
a\in S(m,g),
\intertext{and in \cite{BuTo0}, 
\eqref{Eq:IntrSchattResult1}
it is proved that}
\label{Eq:IntrSchattResult2}
a\in L^p
\quad &\iff \quad
\op ^w(a) \in \mascI _p,
\quad \text{when}\quad
h_g^{k/2}m\in L^p,\ a\in S(m,g).
\end{align}
We observe that
\eqref{Eq:IntrSchattResult1} deals with
Schatten-von Neumann properties for 
the whole symbol class $S(m,g)$, while
\eqref{Eq:IntrSchattResult2} is focused
on more individual symbols. In the case
$p\in (0,1]$, the right implication 
\begin{equation}\label{Eq:IntrSchattResult1R}
m\in L^p
\quad \Longrightarrow \quad
\op ^w(a) \in \mascI _p,
\quad \text{when}\quad
a\in S(m,g),
\end{equation}
in \eqref{Eq:IntrSchattResult1} was proved in
\cite{Tof14}. We also remark
that the right implication 
\begin{equation}\label{Eq:IntrSchattResult2R}
a\in L^p
\quad \Longrightarrow \quad
\op ^w(a) \in \mascI _p,
\quad \text{when}\quad
h_g^{k/2}m\in L^p,\ a\in S(m,g).
\end{equation}
in \eqref{Eq:IntrSchattResult2} was deduced
already in \cite{Ho2} in the
case $p=1$, and in \cite{Tof6} for general
$p\in [1,\infty]$. For $p\le 2$, it
suffices to assume that $g$ should be feasible
instead of strongly feasible, in order for
\eqref{Eq:IntrSchattResult1R}
and
\eqref{Eq:IntrSchattResult2R} to hold.

\par

In the paper, we improve \eqref{Eq:IntrSchattResult1R} 
and obtain a version of \eqref{Eq:IntrSchattResult2R} in 
the case $p\in(0,1]$, by introducing
Wiener-Lebesgue spaces 
$\WL_{g}^{q,p}$ with respect to a slowly varying metric 
$g$. By replacing $L^p$ with $\WL ^{1,p}_g$ in
\eqref{Eq:IntrSchattResult1R}
and
\eqref{Eq:IntrSchattResult2R}, we obtain
stronger results than in previous investigations,
because we neither need to assume that $m$ is
$g$-continuous nor $(\sigma ,g)$-temperate.
At first glance, it might seem that we are more 
restrictive since $\WL ^{1,p}_g$ is contained
in $L^p$ when $p\in (0,1]$. However, if
in addition $m$ is $g$-continuous, which is the
case in \cite{Tof14}, then $m\in L^p$,
if and only if $m\in \WL ^{1,p}_g$. (See
Lemma \ref{Lemma:WeightNormEquiv}.)
Since there are no prior investigations of
$\WL^{q,p}_g$-spaces, a significant part of
the paper is devoted to their study.

\medspace

The paper is organized as follows.
In Section \ref{sec1}, we recall
definitions and some facts on symplectic vector spaces,
pseudo-differential operators, the symbol class 
$S(m,g)$, and Schatten-von Neumannn classes.
Here, we also introduce the Wiener-Lebesgue spaces
$\WL ^{q,p}_{g,\theta}$.
%
%
%
%

\par

In Section \ref{sec2}, we examine the 
structure of the $\WL ^{q,p}_{g}$-spaces, or even
more general $\WL ^{q,p}_{g,\theta}$-spaces.
We deduce some invariance properties. We also
show that $\WL ^{q,p}_{g}$ is essentially
increasing with respect to the slowly varying metric $g$.

\par

In Section \ref{sec3}, we employ the results from Section \ref{sec2} to draw conclusions about Schatten-$p$
properties of pseudo-differential operators on $L^2$. 
Section \ref{sec31} is devoted to the standard 
Hörmander-Weyl calculus and in Section \ref{sec32} we 
restrict ourselves to split metrics $g$ in order to find 
analogous results for more general pseudo-differential 
calculi.

\par

Lastly, in Section \ref{sec4} we apply
our results to pseudo-differential operators
with SG symbols or Shubin symbols.

\par

\section*{Acknowledgement}
The first and second authors have been partially 
supported by INdAM - GNAMPA Project 
CUP\_E53C22001930001 (Sc. Resp. S. Coriasco).
The first and second authors gratefully
acknowledge also the support by the
Department of Mathematics, Linn{\ae}us 
University, V{\"a}xj{\"o}, Sweden,
during their stay in A.Y. 
2023/2024, when most of the results presented in 
this paper have been obtained.
The third and the forth authors were supported by 
Vetenskapsr{\aa}det (Swedish Science Council)
within the project 2019-04890.

\par

\section{Preliminaries}\label{sec1}

\par

In this section we recall some facts on
symplectic vector spaces and the symplectic
Fourier transform. Thereafter we focus on
the H{\"o}rmander symbol classes $S(m,g)$,
pseodo-differential operators and
Schatten-von Neumann operators,
and recall some basic facts for them.
In the last part of the section we
introduce Wiener-Lebesgue spaces
$W^{q,p}_{g,\theta}(W)$, and discuss some
basic properties.

\par

\subsection{Integrations on real
vector spaces}\label{sec11}

\par

Let $V$ be a real vector space of dimension
$d$, with basis $e_1,\dots ,e_d$, and let
$V'$ be its dual, with dual basis $\ep _1,\dots ,\ep _d$.
In particular,
$$
\scal {e_j}{\ep _k}=\delta _{jk},
$$
where $\scal \cdo \cdo = {\scal \cdo \cdo}_{V,V'}$
is the dual form between $V$ and $V'$. For any
$f\in L^1(V)$, we put
$$
\int _Vf\, dx
\equiv
\idotsint _{\rr d}
f(x_1e_1+\cdots +x_de_d)\, dx_1\cdots dx_d.
$$

\par

For any $f\in L^1(V)$, we define the
Fourier transform by
$$
(\mascF f)(\xi ) = \widehat f(\xi )
\equiv
(2\pi )^{-\frac d2}
\int _V f(x)e^{-i\scal x\xi}\, dx,
\qquad
\xi \in V'.
$$
It follows that $\mascF$ restricts to a homeomorphism
from $\mascS (V)$ to $\mascS (V')$, which
in turn is uniquely extendable to a homeomorphism 
from $\mascS '(V)$ to $\mascS '(V')$, and to
a unitary map from $L^2(V)$ to $L^2(V')$.

\par

\subsection{Symplectic vector spaces}\label{sec12}
The real
vector space $W$ of dimension $2d<\infty$ is called
\emph{symplectic} with symplectic form $\sigma$,
if $\sigma$ is a non-degenerate anti-symmetric bilinear
form on $W$, i.{\,}e. $\sigma
(X,Y)=-\sigma (Y,X)$ for every $X,Y\in W$, and if $\sigma (X,Y)=0$ for
every $Y\in W$, then $X=0$. The coordinates $X=(x,\xi )$ are called
symplectic if the corresponding basis
$e_1,\dots ,e_d,\ep _1,\dots ,\ep _d$
is symplectic, i.{\,}e. it satisfies
$$
\sigma (e_j,e_k)=\sigma (\ep _j,\ep _k)=0,\quad \sigma
(e_j,\ep _k)=-\delta _{jk},\quad  j,k=1,\dots
,d.
$$
It follows that $W$ in a canonical way
may be identified with $\rr
d\oplus \rr d=\rr {2d}$, and that $\sigma$ is given by
\begin{equation}\label{sympform}
\sigma (X,Y)=\scal y\xi -\scal x\eta ,\quad
X=(x,\xi )\in W, \quad Y=(y,\eta )\in W.
\end{equation}
Here $\scal \cdo \cdo$ is the usual scalar
product on $\rr d$. Moreover, let $\pi _1$ and
$\pi _2$ be the projections $\pi _1(x,\xi )=x$
and $\pi _2(x,\xi )=\xi $ respectively, and set
$V=\pi _1W$ and $V'=\pi _2W$, which are identified
with $\sets {(x,0)\in
W}{x\in V}$ and $\sets {(0,\xi )\in W}{\xi \in V'}$ respectively. Then
the dual space of $V$ may be identified with $V'$ through the
symplectic form $\sigma$, and $W$ agrees with the cotangent
bundle (or phase space) $T^*V=V\oplus V'$.

\par

On the other hand, if $V$ is a vector space of dimension $d<\infty$
with dual space $V'$ and duality $\scal \cdo \cdo$, then $W=V\oplus
V'$ is a symplectic vector space with symplectic form given by
\eqref{sympform}.

\par

A linear map $T$ on $W$ is called symplectic if
$\sigma (TX,TY)=\sigma
(X,Y)$ for every $X,Y\in W$. For each pairs of symplectic
bases $e_1,\dots ,e_d,\ep _1,\dots ,\ep _d$
and
$\tilde e_{1},\dots ,\tilde e_{d},
\tilde \ep _{1},\dots ,\tilde \ep _{d}$,
there is a unique linear symplectic
map $T$ such that $Te_j=\tilde e_{j}$ and
$T\ep _j=\tilde \ep _{j}$
for every $j=1,\dots ,d$.
On the other hand, if $T$ is linear and
symplectic and $e_1,\dots ,e_d,\ep _1,\dots ,\ep _d$ is a 
symplectic basis, then
$Te_{1},\dots ,T\ep _d$ is also a symplectic
basis. Consequently, there is a one-to-one relation between
linear symplectic mappings, and representations of $W$ as 
cotangent boundles $T^*V$. We refer to
\cite{Ho3} for more facts about symplectic
vector spaces.

\par

The symplectic volume form is defined by $dX =\sigma
^n/d{!\,}$, and if $U\subseteq W$ is measurable, then $|U|$ 
denotes
the measure of $U$ with respect to $dX$. This implies that
$$
\int _Wa(X)\, dX =\idotsint _{\! \rr {d}\oplus \rr d}
a(x_1e _1+\cdots +\xi _d \ep_{d})\, dx_1\cdots d\xi _d 
$$
is independent of the choice
of the symplectic coordinates $X=(x,\xi
)$ when $f\in L^1(W)$. Consequently, $\mathscr D'(W)$ and
its usual subspaces only depend on
$\sigma$ and are independent of the choice of symplectic
coordinates.

\par

The symplectic Fourier transform $\mathscr F_\sigma$ on $\mathscr
S(W)$ is defined by the formula
$$
\mathscr F_\sigma a(X)=\widehat
a(X) \equiv \pi ^{-n}\int _Wa(Y)e^{2i\sigma (X,Y)}\, dY,
$$
when $a\in \mathscr S(W)$. Then $\mathscr F_\sigma$ is a homeomorphism
on $\mathscr S(W)$ which extends to a homeomorphism on $\mathscr
S'(W)$, and to a unitary operator on $L^2(W)$.
Moreover, $\mathscr
F_\sigma ^2$ is the identity operator.
Note also that $\mathscr
F_\sigma$ is defined without any reference to symplectic
coordinates. By straight-forward computations it follows that
$$
\mathscr F_\sigma (a*b)(X) = \pi ^d\widehat a(X)\widehat b(X),\quad
\mathscr F_\sigma (ab)(X)= \pi ^{-d}(\widehat a *\widehat b)(X),
$$
when $a\in \mathscr S'(W)$, $b\in \mathscr S(W)$, and $*$ denotes the
usual convolution. We refer to
\cite{Fo,Tof3} for more facts
about the symplectic Fourier transform.

\medspace


\par



\subsection{Symbol classes and feasible metrics}\label{sec13}

\par

Next we recall the definition of the symbol classes.
(See \cite{Ho1,Ho2,Ho3}.) Let $N\ge 0$ be an integer,
$V$ be a finite-dimensional vector 
space, $a$ belongs to $\mascC ^N(V)$, the set of
continuously differentiable functions of order $N$,
$g$ be an arbitrary Riemannian 
metric on $V$, and let $0< m\in L^\infty _{loc}(V)$.
For each $k=0,\dots ,N$, let
\begin{equation}\label{e1.1}
|a|^g_k(x)
\equiv
\sup |a^{(k)}(x;y_1,\dots ,y_k)|,
\end{equation}
where the supremum is taken over all
$y_1,\dots ,y_k\in V$ such that
$g_x(y_j)\le 1$ for every $j=1,\dots ,k$. Also set
\begin{equation}\label{e1.2}
\nm a{N,m}^g\equiv \sum _{k=0}^N\sup _{x\in V}\big
(|a|^g_k(x)/m(x) \big ).
\end{equation}

\par

We let $S_N(m,g)$ be the set of all
$a\in \mascC ^N(V)$ such that $\nm a{N,m}^g$
is finite. Also set
\begin{align*}
S(m,g) &= S_\infty {(m,g)}
\equiv
\bigcap _{N\ge 0}S_N(m,g).
\end{align*}
It follows that $S_N(m,g)$ is a Banach space and
$S(m,g)$ is a Fr{\'e}chet space.

\par

In our applications, $V$ here above agrees with the
symplectic vector space $W$, and
$S_N(m, g)$ when $0\le N\le \infty$ are the symbol
classes for the Weyl operators.

\par

Next we recall some properties for the weight
function $m$ and the
metric $g$ on $W$. It follows from Section
18.6 in \cite{Ho3} that for each fixed $X\in W$,
there are symplectic coordinates $Z=(z,\zeta )$
which diagonalize $g_X$, i.{\,}e. $g_X$ takes the form
\begin{equation}\label{e1.0}
\begin{gathered}
g_X(Z) =\sum _{j=1}^d\lambda _j(X)(z_j^2+\zeta _j^2),\quad Z=(z,\zeta
)\in W,
\end{gathered}
\end{equation}
where
\begin{equation}\label{e1.00}
\lambda _1(X)\ge \lambda _2(X)\ge \cdots \ge \lambda _d(X)>0
\end{equation}
only depend on $g_X$ and are independent of the choice of symplectic coordinates which
diagonalize $g_X$.

\par

The \emph{dual metric} $g^\sigma$
and \emph{Planck's function} $h_g$ with respect
to $g$ and the
symplectic form $\sigma$ are defined by
$$
g^\sigma _X(Z)\equiv \sup _{Y\neq 0}\Big ( \frac {\displaystyle{\sigma
(Y,Z)^2}}{\displaystyle{g_X(Y)}}\Big ) \quad \text{and}\quad h_g(X)
=\sup _{Z\neq 0} \Big ( \frac {\displaystyle
{g_X(Z)}}{\displaystyle {g_X^\sigma (Z)}}\Big )^{1/2}
$$
respectively. It follows that if \eqref {e1.0}
and \eqref {e1.00} are fulfilled, then $h_g(X)=\lambda _1(X)$ and
\begin{equation}\tag*{(\ref{e1.0})$'$}
\begin{gathered}
g^\sigma _X(Z) =\sum _{j=1}^d\lambda _j(X)^{-1}(z_j^2+\zeta
_j^2),\quad Z=(z,\zeta )\in W.
\end{gathered}
\end{equation}
We usually assume that
\begin{equation}\label{Eq:UnsertPrincMetric}
h_g(X)\le 1
\quad \iff \quad
g_X\le g_X^\sigma ,\qquad X\in W,
\end{equation}
i.{\,}e. the \emph{uncertainly principle} holds.

\par

The metric $g$ is called \emph{symplectic} if $g_X=g^\sigma _X$ for
every $X\in W$. It follows that $g$ is symplectic if and only if
$\lambda _1(X)=\cdots =\lambda _d(X)=1$ in \eqref {e1.0}. If $g_X$ is given by \eqref{e1.0}, then the corresponding
\emph{symplectic metric} is given by
$$
g^{{}_0} _X(Z) =\sum _{j=1}^d(z_j^2+\zeta_j^2).
$$
We observe that $g^{{}_0}$ is defined in a
symplectically invariant way
(cf. \cite{Tof6}).

\par

Let $X\in W$ be fixed, and let $g=g_X$ be as above. Then the
operator $\Delta _g$ is defined by
$\mathscr F_\sigma (\Delta _gf)=-4g^\sigma \cdot \widehat f$ when
$f\in \mathscr S'(W)$. The operator $\Delta _g$ is related to the
Laplace-Beltrami operator for $g$, and is obviously symplectically
invariantly defined, since similar facts hold for $\mathscr F_\sigma$
and $g^\sigma$. If $Z=(z,\zeta )$ are symplectic coordinates such that
\eqref{e1.0} holds, then it follows by straight-forward computation
that
$$
\Delta _{g_X} =\sum _{j=1}^d \lambda _j(X)^{-1}({\partial _{z_j}^2} +
{\partial _{\zeta _j}^2}).
$$

\par

The Riemannian metric $g$ on $W$ is called
\emph{slowly varying}
if there are positive constants $c$ and $C$ such that
\begin{equation}\label{Eq:Slowly}
\begin{gathered}
g_X(Y-X)\le c\quad \Rightarrow \phantom {\quad 
\text{and}\quad X,Y\in
W}
\\
\quad {C}^{-1}g_Y(Z) \le g_X(Z)\le Cg_Y(Z)\quad \text{for
every}\quad Z\in W.
\end{gathered}
\end{equation}

\par

If $g$ and $G$ are Riemannian metrics, then $G$ is called \emph{$g$-continuous}, if there are
positive constants $c$ and $C$ such that
\begin{equation}\label{Eq:gContMetric}
\begin{gathered}
g_X(Y-X)\le c\quad \Rightarrow \phantom {\quad 
\text{and}\quad X,Y\in
W}
\\
\quad {C}^{-1}G_Y(Z) \le G_X(Z)\le CG_Y(Z)\quad \text{for
every}\quad Z\in W.
\end{gathered}
\end{equation}
Lastly, a positive function $m$ is called
\emph{$g$-continuous} if there are
positive constants $c$ and $C$ such that
\begin{equation}\label{Eq:gContFunc}
\begin{gathered}
g_X(Y-X)\le c\quad \Rightarrow \phantom {\quad 
\text{and}\quad X,Y\in
W}
\\
\quad {C}^{-1}m(Y) \le m(X)\le Cm(Y).
\end{gathered}
\end{equation}

\par

The metric $g$ is called
\emph{$\sigma$-temperate},
if there are positive constants $c$, $C$,
and $N$ such that
$$
g_Y(Z) \le g_X(Z)(1+g_Y^\sigma (X-Y))^N,
\qquad
X,Y,Z\in W.
$$
As in \cite{BuTo0,Tof6}, $g$
is called \emph{feasible}
if it is slowly varying and satisfies 
\eqref{Eq:UnsertPrincMetric}, and 
\emph{strongly feasible} if
it is feasible and $\sigma$-temperate.

\par

The weight function $m$ is called
\emph{$(\sigma ,g)$-temperate},
if there are positive constants $c$, $C$,
and $N$ such that
$$
m(Y) \le m(X)(1+g_Y^\sigma(X-Y))^N,
\qquad
X,Y\in W.
$$

\par

\subsection{An extended family of
pseudo-differential calculi}\label{sec14}

\par

Next we discuss some issues in pseudo-differential
calculus. Let $V$ be a real vector space
of dimension $d$ and $a\in \mascS 
(V\times V')$ be fixed. Suppose also that
$A$ belongs to $\maclL (V)$, the set of all
linear mappings on $V$.
Then the pseudo-differential operator
$\op _A(a)$ is the linear and continuous
operator on $\mascS (V)$, given by
\begin{equation}\label{e0.5}
(\op _A(a)f)(x)
=
(2\pi  ) ^{-d}\iint _{V\times V'}
a(x-A(x-y),\xi )f(y)e^{i\scal {x-y}\xi }\, dyd\xi ,
\end{equation}
when $f\in \mascS(V)$. For
general $a\in \mascS'(V\times V')$, the
pseudo-differential operator $\op _A(a)$
is defined as the linear and
continuous operator from
$\mascS(V)$ to $\mascS'(V)$ with
distribution kernel given by
\begin{equation}\label{atkernel}
K_{a,A}(x,y)=(2\pi )^{-\frac d2}
(\mascF _2^{-1}a)(x-A(x-y),x-y).
\end{equation}
Here $\mascF _2F$ is the partial Fourier transform of 
$F(x,y)\in
\mascS'(V\times V)$ with respect to the $y$ variable.
This definition makes sense, since the mappings
\begin{equation}\label{homeoF2tmap}
\mascF _2\quad \text{and}\quad F(x,y)
\mapsto F(x-A(x-y),x-y)
\end{equation}
are homeomorphisms on $\mascS'(V \times V')$ and
on $\mascS '(V\times V)$, respectively.
In particular, the map $a\mapsto K_{a,A}$
is a homeomorphism from
$\mascS'(V \times V')$ to $\mascS '(V\times V)$.

\par

An important special case appears when
$A=t\cdot I$, with
$t\in \mathbf R$. Here and in what follows,
$I=I_V$ is the identity map on $V$.
In this case we set
$$
\op _t(a) = \op _{t\cdot I}(a).
$$
The \emph{normal} or
\emph{Kohn-Nirenberg representation}, 
$a(x,D)$, is  obtained
when $t=0$, and the \emph{Weyl quantization},
$\op ^w(a)$, is obtained when $t=\frac 12$.
That is,
$$
a(x,D) = \op _0(a)
\quad \text{and}\quad \op ^w(a) = \op _{1/2}(a).
$$

\par

We recall that if $A\in \maclL (V)$, then
it follows from the kernel theorem of Schwartz
and Fourier's inversion formula that
the map $a\mapsto \op _A(a)$ is bijective
from $\mascS '(V\times V')$
to the set of linear and continuous mappings
from $\mascS (V)$ to
$\mascS '(V')$ (cf. e.{\,}g. \cite{Ho1,Tof12}).
We refer to \cite{Ho3,Tof12} for the proof
of the following result, concerning transitions
between different pseudo-differential calculi.

\par

\begin{prop}\label{Prop:CalculiTransfer}
Let $a_1,a_2\in \mascS '(V\times V')$ and
$A_1,A_2\in \maclL (V)$.
Then
\begin{equation}\label{calculitransform}
\op _{A_1}(a_1) = \op _{A_2}(a_2) \quad \iff \quad
e^{i\scal {A_2D_\xi}{D_x }}a_2(x,\xi )=e^{i\scal {A_1D_\xi}{D_x }}a_1(x,\xi ).
\end{equation}
\end{prop}

\par

Note here that the latter equality in 
\eqref{calculitransform} makes sense
since it is equivalent to
$$
e^{i\scal {A_2x}{\xi}}\widehat a_2(\xi ,x)
=e^{i\scal {A_1x}{\xi}}\widehat a_1(\xi ,x),
$$
and that the map $a\mapsto e^{i\scal {Ax} \xi }a$ 
is continuous on
$\mascS '(V\times V')$ (cf. e.{\,}g. \cite{Tof12}).

\par

For any $A\in \maclL(V)$, the
$A$-product, $a\wpr _Ab$ between
$a\in \mascS' (V\times V')$
and $b\in \mascS'(V\times V')$
is defined by the formula
\begin{equation}\label{wprtdef}
\op _A(a\wpr _A b) = \op _A(a)\circ \op _A(b),
\end{equation}
provided the right-hand side makes sense as a continuous 
operator from
$\mascS (V)$ to $\mascS '(V)$. Since the Weyl
case is especially important, we write $\wpr$
instead of $\wpr _A$ when $A=\frac 12I_V$.

\par

We shall mainly consider pseudo-differential
operators with symbols in $S(m,g)$. This
family of operators possesses
several convenient properties. For example, 
suppose that
$g$ is strongly feasible, $m_k$ is
$g$-continuous and $(\sigma ,g)$-temperate,
and that $a_k\in S(m_k,g)$, $k=1,2$. Then
there is a unique $a\in S(m_1m_2,g)$ such
that
$$
\op ^w(a_1)\circ \op ^w(a_2)
=
\op ^w(a).
$$
That is,
\begin{equation}\label{Eq:CompWeylHormClasses}
S(m_1,g)\wpr S(m_2,g) \subseteq S(m_1m_2,g).
\end{equation}

\par




\par

\subsection{Schatten-von Neumann classes}\label{sec15}

\par

In order to discuss full range of
Schatten-von Neumann classes, we recall the
definition of quasi-Banach spaces.

\par

\begin{defn}
A \emph{quasi-norm}
$\nm \cdo{\maclB}$ of order $p\in (0,1]$,
or a \emph{$p$-norm}, to the vector space
$\maclB$, is a functional on $\maclB$
such that the following is true:
\begin{enumerate}
\item $\nm f{\maclB} \ge 0$, when
$f\in \maclB$, with equality only for $f=0$;

\vrum

\item $\nm {\alpha f}{\maclB} =
|\alpha |\, \nm {f}{\maclB}$, when
$f\in \maclB$ and $\alpha \in \mathbf C$;

\vrum

\item $\nm {f+g}{\maclB}^p
\le
\nm f{\maclB}^p+\nm g{\maclB}^p$, when
$f,g \in \maclB$.
\end{enumerate}
We equip $\maclB$ with the topology induced
by $\nm \cdo{\maclB}$.
The space $\maclB$ is called a
\emph{quasi-Banach space of order $p$}, or a
\emph{$p$-Banach space}, if $\maclB$ is complete
under this topology.
\end{defn}

\par

Evidently, a topological
vector space is a Banach space, if and only if
it is a quasi-Banach space of order $1$. 

\medspace

Let $\maclH _1$ and $\maclH _2$ be
Hilbert spaces, and let $T$ be a linear
map from $\maclH _1$ to $\maclH _2$. For
every integer $j\ge 1$, the \emph{singular number}  of $T$ of
order $j$ is given by
$$
\sigma _j(T) = \sigma _j(\maclH _1,\maclH _2,T)
\equiv \inf \nm {T-T_0}{\maclH _1\to \maclH _2},
$$
where the infimum is taken over all linear operators $T_0$ from $\maclH _1$
to $\maclH _2$ with rank at most $j-1$. Therefore, $\sigma _1(T)$
equals $\nm T{\maclH _1\to \maclH _2}$, while
$\sigma _j(T)$ is non-negative and
decreases with $j$.

\par

For any $p\in (0,\infty ]$ we set
$$
\nm T{\mascI _p} = \nm T{\mascI _p(\maclH _1,\maclH _2)}
\equiv \nm { \{ \sigma _j(\maclH _1,\maclH _2,T) \} _{j=1}^\infty}{\ell ^p}
$$
(which might attain $+\infty$). The operator $T$ is called a \emph{Schatten-von
Neumann operator} of order $p$ from $\maclH _1$ to $\maclH _2$, if
$\nm T{\mascI _p}$ is finite, i.{\,}e.
$\{ \sigma _j(\maclH _1,\maclH _2,T) \} _{j=1}^\infty$ should belong to $\ell ^p$.
The set of all Schatten-von Neumann operators of order $p$ from
$\maclH _1$ to $\maclH _2$ is denoted by $\mascI _p =
\mascI _p(\maclH _1,\maclH _2)$. We note that
$\mascI _\infty(\maclH _1,\maclH _2)$ agrees with $\maclB (\maclH _1
,\maclH _2)$ (also in norms), the set of linear and bounded operators
from $\maclH _1$ to $\maclH _2$. If $p<\infty$, then
$\mascI _p(\maclH _1,\maclH _2)$ is contained in $\maclK(\maclH _1
,\maclH _2)$, the set of linear and compact operators from $\maclH _1$
to $\maclH _2$. The spaces $\mascI _p(\maclH _1,\maclH _2)$ for
$p\in (0,\infty ]$ and $\maclK(\maclH _1 ,\maclH _2)$ are quasi-Banach
spaces which are Banach spaces when $p\ge 1$. Furthermore,
$\mascI _2(\maclH _1,\maclH _2)$ is a Hilbert space and agrees with the
set of Hilbert-Schmidt operators from $\maclH _1$ to $\maclH _2$ (also in
norms). We set $\mascI  _p(\maclH )=\mascI  _p(\maclH ,\maclH )$.

\par

The set $\mascI _1(\maclH _1,\maclH _2)$ is the set of trace-class
operators from $\maclH _1$ to $\maclH _2$, and $\nm \cdo
{\mascI _1 (\maclH _1,\maclH _2)}$ coincides with the trace-norm. If in addition
$\maclH _1=\maclH _2=\maclH$, then the trace
$$
\operatorname{Tr}_\maclH (T) \equiv \sum _{\alpha} (Tf_\alpha ,f_\alpha)_{\maclH}
$$
is well-defined and independent of the orthonormal basis $\{ f_\alpha \}_{\alpha}$
in $\maclH$.

\par

Now let $\maclH _3$ be another Hilbert space
and let $T_k$ be a linear and continuous operator from
$\maclH _k$ to $\maclH _{k+1}$, $k=1,2$. Then we recall 
the H{\"o}lder relation 
\begin{equation}\label{SchattenComp}
\begin{gathered}
\nm {T_2\circ T_1}{\mascI _{r}(\maclH _1,\maclH _3)}\le 
\nm {T_1}{\mascI _{p_1}(\maclH _1,\maclH _2)}
\nm {T_2}{\mascI _{p_2}(\maclH _2,\maclH _3)}
\\[1ex]
\text{when}\quad
\frac 1{p_1}+\frac 1{p_2}=\frac 1r
\end{gathered}
\end{equation}
(cf. e.{\,}g. \cite{Si,Tof10}).

\par

In particular, the map $(T_1,T_2)\mapsto T_2^*\circ T_1$ is continuous from 
$\mascI _p(\maclH _1,\maclH _2)\times \mascI _{p'}(\maclH _1,\maclH _2)$
to $\mascI _1(\maclH _1)$, giving that
\begin{equation}\label{Eq:SchattenScalar}
(T_1,T_2)_{\mascI _2(\maclH _1,\maclH _2)}\equiv
\operatorname{Tr}_{\maclH _1}(T^*_2\circ T_1)
\end{equation}
is well-defined and continuous from $\mascI _p(\maclH _1,\maclH _2)\times
\mascI _{p'}(\maclH _1,\maclH _2)$ to $\mathbf C$. If $p=2$, then
the product, defined by \eqref{Eq:SchattenScalar} agrees with the scalar product in
$\mascI _2(\maclH _1,\maclH _2)$.

\par

The proof of the following result is omitted, since 
it can be found in e.{\,}g.
\cite{BirSol,Si}.

\par

\begin{prop}\label{Prop:SchattenDual}
Let $p\in [1,\infty]$, $\maclH _1$ and $\maclH _2$ be Hilbert spaces, and let $T$
be a linear and continuous map from $\maclH _1$ to $\maclH _2$. Then the
following is true:
\begin{enumerate}
\item if $q\in [1,p' ]$, then
$$
\nm T{\mascI _p(\maclH _1,\maclH _2)}
=\sup |(T,T_0)_{\mascI _2(\maclH _1,\maclH _2)}|,
$$
where the supremum is taken over all $T_0\in \mascI
_q(\maclH _1,\maclH _2)$ such that $\nm {T_0}{\mascI
_{p'}(\maclH _1,\maclH _2)}\le 1$;

\vrum

\item if in addition $p<\infty$, then the dual of $\mascI
_p(\maclH _1,\maclH _2)$ can be identified through the form
\eqref{Eq:SchattenScalar}.
\end{enumerate}
\end{prop}

\par

Later on we are especially interested in finding necessary and sufficient
conditions on symbols, in order for the corresponding pseudo-differential
operators to belong to $\mascI _p(\maclH _1,\maclH _2)$, where
$\maclH _1$ and $\maclH _2$ satisfy
$$
\mascS (V)\hookrightarrow \maclH _1,\maclH _2
\hookrightarrow \mascS'(V).
$$
Therefore, for such Hilbert spaces and $p\in (0,\infty ]$, let
\begin{align}
s_{A,p}(\maclH _1,\maclH _2)
&\equiv
\sets {a\in \mascS '(V\times V')}
{\op _A(a)\in \mascI _p(\maclH _1,\maclH_2 )}\notag
\intertext{and}
\nm a{s_{A,p}(\maclH _1,\maclH _2)}
&\equiv \nm {\op _A(a)}{\mascI _p(\maclH _1,\maclH _2)}.
\label{SchattNormId}
\end{align}
Since the map $a\mapsto \op _A(a)$ is bijective from 
$\mascS'(V \times V')$ to the set of all linear and 
continuous operators from $\mascS(V)$ to 
$\mascS'(V)$, it follows from the 
definitions
that the map $a\mapsto \op _A(a)$ restricts to a
bijective and isometric map
from $s_{A,p}(\maclH _1,\maclH _2)$ to
$\mascI _p(\maclH _1, \maclH _2)$. We put
\begin{alignat*}{3}
s_{A,p}(W) &= s_{A,p}(\maclH _1,\maclH _2) &
\quad &\text{when} &\quad
\maclH _1 &= \maclH_2 =L^2(V).
\end{alignat*}

\par

For convenience we also 
put $s_p^w = s_{A,p}$ in the Weyl case (i.{\,}e.
when $A=\frac 12\cdot I_V$).

\par

\subsection{Wiener Lebesgue spaces with
respect to slowly varying metrics}\label{sec16}

\par

Before defining the Wiener-Lebesgue spaces,
we recall some facts about $g$-balls, which
are given by
\begin{equation}\label{Eq:gBalls}
U_{X,R} = U_{g,X,R}
\equiv \sets {Y\in W}{g_{X}(Y-X)< R^2},
\end{equation}
when $X\in W$ and $R>0$.
The following lemma is a consequence of
Lemma 1.4.9 and the proof of Theorem 1.4.10
in \cite{Ho3}. The proof is therefore omitted.

\par

\begin{lemma}\label{Lemma:AdmCover}
Let $g$ be slowly varying on $W$ and let $c$ and $C$
be as in \eqref{Eq:Slowly}. Then there exists
a sequence $\{X_j\} _{j=1}^\infty$ such that if
$$
U_j = U_{X_j,R},
$$
for some $R>0$ such that $\frac c2<R^2<c$,
then the following is true:
\begin{enumerate}
\item $g_{X_j}(X_j-X_k)\ge \frac c{2C}$ for every $j,k=1,2,\dots$ such that $j\neq k$;

\vrum

\item $W =\bigcup _{j=1}^\infty U_j$;

\vrum

\item if $j\in\mathbf{Z}_+$ is fixed, then
$U_j\cap U_k\neq \emptyset$ for at
most $(4C^3+1)^{2d}$ numbers of $k$.
\end{enumerate}
\end{lemma}

\par

\begin{defn}\label{Def:AdmCover}
Let $g$ be slowly varying on $W$, $c$ and $C$
be as in \eqref{Eq:Slowly}. Then the family
of $g$-balls $\{ U_j\} _{j=1}^\infty$ in Lemma 
\ref{Lemma:AdmCover} is called
an \emph{admissible $g$-covering} of $W$.
\end{defn}


\par

\begin{rem}\label{Rem:AdmCover1}
Let $\{ X_j\} _{j=1}^\infty$ be as in Lemma 
\ref{Lemma:AdmCover}. For future reference, we 
observe that if $Y\in W$, $r,R_1,R_2>0$ satisfy
$$
\frac c2 <R_1^2<R_2^2<c,
\quad 
r < \frac{R_2-R_1}{2C},
$$
and
$U_{Y,r} \cap U_{X_j,R_1}\neq \emptyset$
for some $j\in\mathbf{Z}_+$,
then $U_{Y,r} \subseteq U_{X_j,R_2}$.

\par

As a consequence of Lemma 
\ref{Lemma:AdmCover} there are at most
$(4C^3+1)^{2d}$ numbers of $U_{X_j,R_1}$
or $U_{X_j,R_2}$ which intersect with
$U_{Y,r}$.

\par

In fact, suppose $Z \in U_{Y,r} \cap U_{X_j,R_1}$.
Then, for every $X \in U_{Y,r}$ we have that
\begin{align*}
\big ( g_{X_j}(X-X_j) \big )^{\frac 12}
&=
\big (g_{X_j}(X-Z+Z-X_j)\big )^{\frac 12}
\\[1ex]
&\leq
\big (g_{X_j}(Z-X_j)\big )^{\frac 12}
+
\big (g_{X_j}(X-Z)^{\frac 12}.
\\[1ex]
\intertext{By the fact that $g$ is slowly varying, we 
obtain  that $g_{X_j} \leq C g_Z \leq C^2 g_Y$.
Hence, we have}
\big (g_{X_j}(Z-X_j) \big )^{\frac 12} + \big (g_{X_j}(X-Z) \big )^{\frac 12}
&\leq
R_1 + C \big (g_{Y}(Z-X) \big )^{\frac 12}
\\[1ex]
&\leq R_1 + 2Cr
<R_2,
\end{align*}
which shows that $X \in U_{X_j,R_2}$, and the assertion
follows.
\end{rem}

\par

\begin{defn}\label{Def:WienerLebDef}
Let $p,q\in (0,\infty ]$, $\theta \in \mathbf R$,
$g$ be a slowly varying
metric on $W$, $\{ U_j\} _{j=1}^\infty$ be an
admissible $g$-covering, and let $U\subseteq \rr d$
be an open ball such that $\{ j+U \} _{j\in \zz d}$
covers $\rr d$.
\begin{enumerate}
\item The \emph{Wiener-Lebesgue space}
$\WL ^{q,p} (\rr d)$ (with respect to $p$ and $q$)
consists of all
measurable functions $f$ such that
$\nm f{\WL ^{q,p}}$ is finite, where
$$
\nm f{\WL ^{q,p}}
\equiv
\Nm {\{ \nm f{L^q(j+U)}\}
_{j\in \zz d}}{\ell ^p(\zz d)}.
$$

\vrum

\item The \emph{Wiener-Lebesgue space}
$\WL _{g,\theta}^{q,p} (W)$
(with respect to $p$, $q$, $\theta$ and $g$)
consists of all
measurable functions $a$ such that
$\nm a{\WL _{g,\theta}^{q,p}}$ is finite, where
$$
\nm a{\WL _{g,\theta}^{q,p}}
\equiv
\Nm {\{ \nm {a}{L^q(U_j)}\cdot |U_j|^\theta\}
_{j\in\mathbf{Z}_+}}{\ell ^p(I)}.
$$
\end{enumerate}
\end{defn}

\par


We remark that $\WL _{g,\theta}^{q,p} (W)$ 
is a quasi-Banach space of order $\min (1,p,q)$,
and independent of the choice of admissible
$g$-covering $\{ U_j\} _{j\in\mathbf{Z}_+}$ in
Definition \ref{Def:WienerLebDef}
(cf. Proposition \ref{Prop:WLNormIndep}
below). In particular, it follows that
$\WL ^{q,p} (\rr d)$
is independent of the choice of 
$U$ in Definition \ref{Def:WienerLebDef}.
(This follows from \cite{Gro1} as well.)
If $p,q\ge 1$, then $\WL _{g,\theta}^{q,p}(W)$ 
is a Banach space.

\par

For $p\in (0,1]$ and $q\in (0,\infty ]$,
the choice of parameter $\theta =\frac 1p-\frac 1q$
in the $\WL _{g,\theta}^{q,p}$ spaces is of special interest. For this reason
we let
$$
\WL _{g}^{q,p}=\WL _{g,\theta}^{q,p}
\quad \text{when}\quad
\theta =\frac 1p-\frac 1q.
$$

\par

\section{Structural properties for
Wiener-Lebesgue spaces}\label{sec2}

\par

In this section we show some
basic properties for
$\WL _{g,\theta}^{q,p}$-spaces.
First we show that such spaces are
invariantly defined with respect to
the choice of admissible
$g$-covering. Then we show that
such spaces increase if we replace
the metrics with
corresponding symplectic metrics.

\par

\begin{prop}\label{Prop:WLNormIndep}
Let $p,q\in (0,\infty ]$, $\theta \in \mathbf R$
and $g$ be slowly varying on $W$.
Then $\WL ^{q,p}_{g,\theta}(W)$
is independent of the choice of admissible
$g$-covering $\{ U_j\} _{j\in\mathbf{Z}_+}$ in
Definition \ref{Def:WienerLebDef}.
\end{prop}

\par

\begin{rem}
Since $\WL ^{q,p}_{g,\theta}(W)$ is defined
through quasi-norm estimates, it follows
from Proposition \ref{Prop:WLNormIndep}
that different admissible coverings give
rise to equivalent quasi-norms for
$\WL ^{q,p}_{g,\theta}(W)$.
\end{rem}

\par

\begin{proof}[Proof of Proposition 
\ref{Prop:WLNormIndep}]
We only prove the result when $p\le q<\infty$.
The other cases follow by similar arguments and
are left to the reader.
By considering $b(X)=|a(X)|^q$,
we reduce ourselves to the case when
$q=1$ and $p\le 1$. We may also replace $\theta$
by $\theta /p$.

\par

Let $\maclU = \{ U_j\}_{j\in\mathbf{Z}_+}$ and
$\maclV= \{ V_k\}_{k\in \mathbf{Z}_+}$ be admissible
$g$-coverings, let
$$
\nm a {\maclU}^p = \sum_{j=0}^\infty \left(\int_{U_j} |a(X)|
\, dX \right)^p |U_j|^{\theta},
$$
and let
$$
\nm a {\maclV}^p = \sum_{k=0}^\infty \left( \int_{V_k}  | a(X)| \, dX \right)^p |V_k|^{\theta}.
$$
By \cite[Lemma 18.4.4]{Ho3}, there is
a bounded sequence $\{ \fy _k\} _{k=0}^\infty$ 
in $S(1,g)$ such that $\fy _k\ge 0$,
$\supp \fy _k\subseteq V_k$ for every $k$,
and $\sum _{k=0}^\infty \fy _k = 1$.

\par

We have
\begin{align*}
\nm a {\maclU}^p
&=
\sum_{j=0}^\infty \left(\int_{U_j} |a(X)|
\, dX \right)^p |U_j|^{\theta}
\\[1ex]
&\asymp
\sum_{j=0}^\infty \left( \int_{U_j} \sum_{k=0}^\infty 
|\varphi_k(x) a(X)| \, dX \right)^p |U_j|^{\theta}
\\[1ex]
&\leq
\sum_{j=0}^\infty \sum_{k=0}^\infty \left( \int_{U_j} 
|\varphi_k(x) a(X)| \, dX \right)^p |U_j|^{\theta}
\\[1ex]
&\asymp
\sum_{j=0}^\infty \sum_{k=0}^\infty \left( \int_{U_j} 
|\varphi_k(x) a(X)| \, dX \right)^p |V_k|^{\theta},
\intertext{where the last relation follows
from the fact that
$|U_j|\asymp |V_k|$ when $U_j\cap V_k\neq \emptyset$
in combination with the fact that $g$ is slowly
varying.
Since there is an upper bound of 
intersections
between $U_j$ and $V_k$ in view of
Remark \ref{Rem:AdmCover1}, we obtain}
\nm a {\maclU}^p
&\lesssim
\sum_{k=0}^\infty \left( \sum_{j=0}^\infty
\left( \int_{U_j}  |\varphi_k(x) a(X)| \, dX 
\right)^p |V_k|^{\theta}\right)
\\[1ex]
&\asymp \sum_{k=0}^\infty \left( \int_{W}  |\varphi_k(x) a(X)| \, dX \right)^p |V_k|^{\theta}
\\[1ex]
& \leq
\sum_{k=0}^\infty \left( \int_{V_k}  | a(X)| \, dX \right)^p |V_k|^{\theta}
\\[1ex]
& = \nm a {\maclV}^p.\qedhere
\end{align*}
\end{proof}

\par



\par


\par

Next we show that $\WL _{g}^{q,p}(W)$
is contained in $\WL _{g^{{}_0}}^{q,p}(W)$, when
$g$ is feasible.
For that reason we need the following proposition.

\par

\begin{prop}\label{bnum}
Let $g$ be a slowly varying metric on $W$,
$G$ be a $g$-continuous metric such that $g\le G$,
and let $\{ U_{X_j,R}\} _{j=1}^\infty$ be an
admissible $g$-covering of $W$. Then there
exists an admissible $G$-covering
%
$\{ U_{G,k}\} _{k=1}^\infty$ of
$W$ given by
$$
U_{G,k}=U_{G,Y_k,r}
=
\sets {X\in W}{G_{Y_k}(X-Y_k)<r^2},
\quad
k\in \mathbf Z_+
$$
such that
\begin{equation}\label{Eq:WLgEmb}
C_1 \frac {|U_{X_j,R}|}{|U_{G,X_j,r}|}
\leq
N_j
\leq
C_2 \frac {|U_{X_j,R}|}{|U_{G,X_j,r}|}
\end{equation}
when
$N_j$ is the number of $U_{G,k}$
intersecting $U_{X_j,R}$, and the constants
$C_1,C_2>0$ are independent of $j\in \mathbf Z_+$.
\end{prop}

\par

\begin{proof}
Let $U_{X_j,R_1}$ and $U_{X_j,R_2}$ be as in Remark
\ref{Rem:AdmCover1}. If $r>0$ is chosen small enough,
then there is an admissible $G$-covering of
$W$, given by
$$
U_{G,k} =\sets {X\in W}{G_{Y_k}(X-Y_k)<r^2},
\quad
k\in \mathbf Z_+
$$
such that $U_{G,k} \subseteq U_{X_j,R_2}$ when
$U_{G,k}$ intersects $U_{X_j,R_1}$. The facts that
$G$ is $g$-continuous and $g\le G$ guarantees that such 
$r$ exists. Also, let $\Omega _j$ be the set of
all $k\in \mathbf Z_+$ such that $U_{G,k}$ intersects
$U_j$ and let $N_j=|\Omega _j|$.

\par

We have
$$
U_{X_j,R_1}
\subseteq
\bigcup _{k\in \Omega _j}U_{G,k}
\subseteq
U_{X_j,R_2}.
$$
Since the balls $\{U_{G,k}\}_{k\in I}$ form an admissible
covering of $W$, there is an upper bound $M$ of
overlapping $U_{G,k}$. This gives
$$
\frac 1M\sum _{k\in \Omega _j}|U_{G,k}|
\le |U_{X_j,R_2}|
\quad \iff \quad
\sum _{k\in \Omega _j}|U_{G,k}|\le M|U_{X_j,R_2}|.
$$
Since $G$ is $g$-continuous, we have
$$
C_3|U_{G,X_j,r}|
\le
|U_{G,k}|
\le
C_4|U_{G,X_j,r}|
$$
for some constants $C_3,C_4>0$ which are independent
of $k$. A combination of these estimates gives
$$
{C_3} N_j\, |U_{G,X_j,r}|
=
{C_3} |\Omega _j|\, |U_{G,X_j,r}|
\le 
\sum _{k\in \Omega _j}|U_{G,k}|
\le M|U_{X_j,R_2}|,
$$
which leads to the second inequality in
\eqref{Eq:WLgEmb}.

\par

We also have
$$
|U_{X_j,R_1}|\le
\left |
\bigcup _{k\in \Omega _j}U_{G,k}
\right |
\le
\sum _{k\in \Omega _j}|U_{G,k}|
\le C_4N_j|U_{G,X_j,r}|,
$$
giving the first inequality in \eqref{Eq:WLgEmb},
giving the result.
\end{proof}

\par

Since all $g$-balls are of the same size when
$g$ is symplectic, the previous proposition takes the following form.

\par

\begin{cor}\label{Cor:bnum}
Let $g$ be a feasible metric on $W$, and
let $\{ U_{X_j,R}\} _{j=1}^\infty$ be an
admissible $g$-covering of $W$. Then there
exists an admissible $g^{{}_0}$-covering
$\{ U_k^0\} _{k=1}^\infty$ of
$W$ given by
$$
U_k^0 =\sets {X\in W}{g_{Y_k}(X-Y_k)<r^2},
\quad
k\in \mathbf Z_+
$$
such that $N_j \leq C |U_{X_j,R}|$, where
$N_j$ is the number of $U_k^0$
intersecting $U_{X_j,R}$, and the constant
$C>0$ is independent of $j\in \mathbf Z_+$.
\end{cor}

\par

\begin{prop}\label{Prop:WLgEmb}
Let $g$ be a slowly varying metric on $W$ and let
$G$ be a $g$-continuous metric such that
$g\le G$. Also, suppose that
and $0<p\leq q < \infty$.
Then
$$
\WL _{g}^{q,p}(W)
\subseteq
\WL _G^{q,p}(W).
$$
\end{prop}

\par

\begin{proof}
Let $p_0 = \frac pq \in (0,1]$ and $b(X) = |a(X)|^q$.
The inequalities in \eqref{Eq:WLgEmb} shall
be combined with
\begin{equation}\label{Eq:ConcavitEst}
\sum _{k=1}^Nx_k^{p_0} \le N^{1-p_0}
\left (
\sum _{k=1}^Nx_k
\right )^{p_0},
\qquad
x_1,\dots ,x_N\ge 0,
\end{equation}
which follows by concavity of $t\mapsto t^{p_0}$.

\par

We use the same notations as in
the proof of Proposition \ref{bnum}.
Since $\nm a{\WL^{q,p}_G}^p 
=\nm b{\WL^{1,p_0}_G}^{p_0}$, we obtain
\begin{align*}
\nm a{\WL^{q,p}_G}^p 
&\asymp
\sum _{k=1}^\infty
\left (
\int _{U_{G,k}}|b(X)|\, dX
\right )^{p_0}|U_{G,k}|^{1-p_0}
\\[1ex]
&\le
\sum _{j=1}^\infty
\left (
\sum _{k\in \Omega _j}
\left (
\int _{U_{G,k}}|b(X)|\, dX
\right )^{p_0}|U_{G,k}|^{1-p_0}
\right )
\\[1ex]
&\le
\sum _{j=1}^\infty
\left (
|\Omega _j|^{1-p_0}
\left (
\sum _{k\in \Omega _j}\int _{U_{G,k}}|b(X)|\, dX
\right )^{p_0}|U_{G,k}|^{1-p_0}
\right ),
\end{align*}
where the last inequality follows from
\eqref{Eq:ConcavitEst}. Since there is a bound $M$
of overlapping $U_{G,k}$,
$$
|U_{X_j,R_1}|\asymp |U_{X_j,R_2}|,
\quad \text{and}\quad
|U_{G,k}|\asymp |U_{G,X_j,r}|,
$$
when $U_{G,k}$
intersects with $U_{X_j,R_1}$,
Proposition \ref{bnum} gives
\begin{align*}
\nm a{\WL^{q,p}_G}^p
&\lesssim
\sum _{j=1}^\infty
\left (
\left (
\frac {|U_{X_j,R_2}|}{|U_{G,X_j,r}|}
\right )^{1-p_0}
\left (
\sum _{k\in \Omega _j}\int _{U_{G,k}}|b(X)|\, dX
\right )^{p_0}|U_{G,X_j,r}|^{1-p_0}
\right )
\\[1ex]
&\le
\sum _{j=1}^\infty
\left (
\left (
M\int _{U_{X_j,R_2}}|b(X)|\, dX
\right )^{p_0}|U_{X_j,R_2}|^{1-p_0}
\right )
\\[1ex]
&\asymp
\nm b{\WL ^{1,p_0}_{g}}^{p_0} 
= \nm a{\WL^{q,p}_g}^p\ ,
\end{align*}
and the result follows from these estimates.
\end{proof}

\par

Since $g^{{}_0}$ is $g$-continuous and
$g\le g^{{}_0}$ whenever $g$ is feasible,
the following corollary is an immediate consequence
of Proposition \ref{Prop:WLgEmb}.

\par 

\begin{cor}\label{Cor:WLgEmb}
Let $g$ be feasible on $W$ and $0<p\leq q < \infty$.
Then
$$
\WL _{g}^{q,p}(W)
\subseteq
\WL _{g^{{}_0}}^{q,p}(W).
$$
\end{cor}



\par

\section{Quasi-Banach Schatten-von
Neumann properties in
pseudo-differential calculus}\label{sec3}

\par

In this section we deduce Schatten-von Neumann
properties, with respect to $p\in (0,1]$,
for pseudo-differential operators with symbols
in $S(m,g)$ and with $m$ or $a$ belonging to
$\WL^{1,p}_g(W)$. In Section \ref{sec31} we
deal with Weyl operators, where in the first part
the assumptions on $m$ and $g$ are minimal,
and the operators are acting on $L^2(V)$.
The second part of Section \ref{sec31} is
devoted to operators acting between (different)
Bony-Chemin Sobolev-type spaces $H(m,g)$.
Here, we restrict ourselves and
assume that $m$ and $g$ satisfy the usual
conditions in the Weyl-H{\"o}rmander calculus.
In Section \ref{sec32}, we
consider more general pseudo-differential
calculi, but with some additional restrictions
on $g$.

\par

\subsection{The case of H{\"o}rmander-Weyl
calculus}\label{sec31}

\par

\begin{thm}\label{Thm:SymbClassEmbd}
Let $p\in (0,1]$, $g$ be feasible on $W$, and
$m\in \WL _g^{1,p}(W)$ be a positive function on $W$.
Then $S(m,g)\subseteq s_p^w(W)$.
\end{thm}

\par

For the proof we need the following lemma on
embeddings between $s_p^w(W)$ and
Sobolev-type spaces of distributions
with suitable numbers of derivatives belonging to
$\WL^{1,p}(W)$.

\par

\begin{lemma}\label{Lemma:SchattenWienerSobEst}
Let $p\in (0,1]$. Then there is an integer
$N\ge 1$ and a constant $C>0$ which only
depends on $p$ and the dimension of $W$ such that
$$
\nm a{s_p^w(W)}
\le
C\nm {(1-\Delta )^Na}{\WL ^{1,p}(W)}.
$$
\end{lemma}

\par

\begin{proof}
The symbol $b(X) = (1+|X|^2)^{-N}$ belongs to
$s_p^w(W)$, provided that $N\ge 1$ is chosen large 
enough (see e.{\,}g. \cite[Theorem 2.6]{Tof3}).
It follows that $\fy =\mascF _\sigma b\in s_p^w(W)$,
since
$s_p^w(W)$ is invariant under the symplectic Fourier
transform. This gives
\begin{multline*}
\nm a{s_p^w}
=
\nm {(1-\Delta )^{-N}((1-\Delta )^Na)}{s_p^w}
\\[1ex]
\asymp
\nm {\fy *((1-\Delta )^Na)}{s_p^w}
\lesssim
\nm \fy{s_p^w} \nm {(1-\Delta )^Na}{\WL ^{1,p}}.
\end{multline*}
Here the inequality follows from
\cite[Proposition 5.11]{BimTof}. This gives
the result.
\end{proof}

\par

\begin{proof}[Proof of Theorem
\ref{Thm:SymbClassEmbd}]
By $g\le g^{{}_0}$, Corollary \ref{Cor:WLgEmb},
and the fact that $S(m,g)$ increases with $g$,
it suffices to prove the result with
$g^{{}_0}$ in place of $g$. Hence
we may assume that $g$ is symplectic.

\par

Let $U_j$ and $\fy _k$ be
the same as in the proof of Proposition
\ref{Prop:WLNormIndep}, with $V_k=U_k$.
Also, let $g_j =g_{X_j}$ and
$U_{j,k}=U_j\cap U_k$. By Lemma
\ref{Lemma:SchattenWienerSobEst} and the fact that
$s_p^w(W)$ are invariant under symplectic
transformations we obtain
\begin{align*}
\nm {\fy _ja}{s_p^w}
\le
C\nm {(1-\Delta _{g_j})^N(\fy _ja)}
{\WL ^{1,p}_{g_j}}
\end{align*}
Hence \eqref{Eq:DerWL1pEst2},
$\supp \fy _j\subseteq U_j$, and the fact
that $p\le 1$ give
\begin{align*}
\nm a{s_p^w}^p
&=
\NM {\sum _{j=1}^\infty (\fy _ja)}{s_p^w}^p
\le
\sum _{j=1}^\infty
\nm {\fy _ja}{s_p^w}^p
\\[1ex]
&\lesssim
\sum _{j=1}^\infty
\nm {(1-\Delta _{g_j})^N(\fy _ja)}
{\WL ^{1,p}_{g_j}}^p
\\[1ex]
&\asymp
\sum _{j=1}^\infty \sum _{k=1}^\infty
\left (
\int _{U_{j,k}}|(1-\Delta _{g_j})^N(\fy _j(X)a(X))|\, dX
\right )^p
\\[1ex]
&\lesssim
\sum _{j=1}^\infty \sum _{k=1}^\infty
\sum _{|\alpha |\le 2N}
\left (
\int _{U_{j,k}}|(\partial _{g_j}^\alpha a)(X)|\, dX
\right )^p
\\[1ex]
&\lesssim
\nmm a ^p
\sum _{j=1}^\infty \sum _{k=1}^\infty
\left (
\int _{U_{j,k}}|m(X)|\, dX
\right )^p
\end{align*}
Here $\nmm a$ denotes a semi-norm of $a$ in $S(m,g)$.
Since there is a bound of overlapping $U_j$, it follows
from these estimates that
$$
\nm a{s_p^w}^p
\lesssim
\nmm a ^p
\sum _{j=1}^\infty
\left (
\int _{U_j}|m(X)|\, dX
\right )^p \asymp \nmm a ^p \nm m{\WL _{g}^{1,p}}^p,
$$
which gives the result.
\end{proof}

\par

The next result improves Theorem 
\ref{Thm:SymbClassEmbd}. It also extends
\cite[Theorem 3.9]{Ho2}.

\par

\begin{thm}\label{Thm:MainResult2}
Let $p\in (0,1]$, $g$ be feasible on $W$,
$m$ be a positive function on $W$
such that $h_g^{k/2}m\in \WL _g^{1,p}(W)$
for some $k\ge 0$,
and suppose $a\in S(m,g)\cap \WL _g^{1,p}(W)$. Then
$a\in s^w_{p}(W)$.
\end{thm}

\par

For the proof we need the following lemmas.

\par




\par

\begin{lemma}\label{Lemma:EstIntermedDer}
Let $p\in (0,\infty ]$, $q\in [1,\infty ]$,
$N \in \mathbf N$ and
$f \in \WL ^{q,p}(\rr d) \cap
\mathscr C^N(\rr d)$. Then there
exists a constant $C>0$ such that
\begin{equation}\label{Eq:DerWL1pEst}
\nm
{\partial^\alpha f}{\WL ^{q,p}}^p
\leq C
\left(
\nm f {\WL ^{q,p}}^p
+
\sum_{|\beta|=N}
\nm {\partial^\beta f}{\WL ^{q,p}}^p 
\right).
\end{equation}
\end{lemma}

\par


\par

\begin{lemma}\label{Lemma:m0uppsk}
Let $g$ be a feasible metric on $W$,
$\alpha \in [0,1]$ and set $G=h_g^{-\alpha}g$.
Also, assume that $N\ge 0$ is an integer which
is fixed, $m>0$ is a weight function on
$W$, $a\in \mathscr{C}^N(W)$, and set
$$
m_0=\sum _{n=0}^{N-1}|a|_n^G + h_g^{\alpha N/2}m.
$$
Then the following are true:
\begin{enumerate}
\item if $p\in (0,1]$, then
\begin{equation}\label{m0}
\nm {m_0}{\WL ^{1,p}_g}\le C(\nm a{\WL ^{1,p}_g}
+\nm {h_g^{\alpha N/2}m}{\WL ^{1,p}_g})\text ;
\end{equation}

\vrum

\item if $a\in \WL^{1,p}_g(W)$
and $h_g^{\alpha N/2}m\in \WL ^{1,p}_g(W)$,
then $m_0\in \WL ^{1,p}_g(W)$.
\end{enumerate}
\end{lemma}

\par

\begin{proof}
[Proof of Lemma \ref{Lemma:EstIntermedDer}]
Let $U$ be as in Definition \ref{Def:WienerLebDef}.
Then there
exists a constant $C>0$ such that, for
any $|\alpha| \leq N$ and $j\in \zz d$,
$$
\nm {\partial^\alpha f}{L^q(j+U)}
\leq
C\left(
\nm f {L^q(j+U)} + \sum_{|\beta|=N}
\nm {\partial^\beta f}{L^q(j+U)}
\right).
$$
(See e.{\,}g. \cite{BenSha}.) Hence
for a (possibly new) constant $C>0$, we obtain 
$$
\nm {\partial^\alpha f}{L^q(j+U)}^p
\leq
C \left(
\| f \|_{L^q(j+U)}^p
+
\sum_{|\beta|=N}
\nm {\partial^\beta f}{L^q(j+U)}^p \right ).
$$
Summing up with respect to $j \in \zz d$ we have
$$
\begin{aligned}
\| \partial^\alpha f \|_{\WL ^{q,p}}^p
&=
\sum_{j \in \zz d} \| \partial^\alpha f 
\|_{L^q(j+U)}^p
\\[1ex]
&\leq C
\left(
\sum_{j \in \zz d} \| f \|_{L^q(j+U)}^p + 
\sum_{j \in \zz d} \sum_{|\beta|=N} \| 
\partial^\beta f \|_{L^q(j+U)}^p
\right )
\\[1ex]
&= C \left(
\| f \|_{\WL ^{q,p}}^p +  \sum_{|\beta|=N} \| 
\partial^\beta f \|_{\WL ^{q,p}}^p 
\right).
\quad
\qedhere
\end{aligned}
$$
\end{proof}

\par

\begin{proof}[Proof of Lemma \ref{Lemma:m0uppsk}]
It suffices to prove (i). By \cite[Lemma 6.1]{Tof6},
it follows that
$|a|^G_k\le Cm_0$ for some constant $C>0$.
Let $V_j=U_j$, and let $\fy _j$ and $U_j$ for 
$j\in\mathbf{Z}_+$
be as in the proof of Proposition 
\ref{Prop:WLNormIndep}.
Also, let $\{
\psi _j\} _{j=1}^\infty$ be a bounded sequence in 
$S(1,g)$ such that $\psi
_j\in C_0^\infty (U_j)$ and $\psi _j=1$
in the support of $\fy
_j$. Lastly, let $g_j=g_{X_j}$ and $G_j=G_{X_j}$.
Then
$$
|\fy _ja|^{G_j}_N=h_{g_j}^{\alpha N/2}|
\fy _ja|^{g_j}_N\le
Ch_{g_j}^{\alpha N/2}\psi _jm,
$$
where the constant $C$ is independent of 
$j\in\mathbf{Z}_+$.
For every $j\in \mathbf Z_+$, let $G_j$ define the 
Euclidean structure on
$W$. By Lemma \ref{Lemma:EstIntermedDer}, and the fact 
that $C$ in
\eqref{Eq:DerWL1pEst} is invariant under changes of symplectic
structures on $W$, it follows that
$$
\nm {\, |\fy _ja|^{G_j}_n}{L^1}
\le
C\big ( \nm {\fy _ja}{L^1}
+
\nm {h_{g_j}^{\alpha N/2}\psi _jm}{L^1}\big ),
$$
where the constant $C$ neither depends on
$j\in \mathbf Z_+$ nor on $n\in \{
0,\dots ,N\}$.

\par

We have
\begin{align*}
\nm {\, |a|^{G}_n}{\WL^{1,p}_g}^p
=
\nm {\, |
\sum _{l=1}^\infty \fy _la
|^{G}_n}{\WL ^{1,p}_g}^p
=
\sum _{j=1}^\infty \left ( \int _{U_j} 
|
\sum _{l=1}^\infty
\fy _l a |^{G}_n(X)\, dX
\right )^p|U_j|^{1-p}.
\end{align*}
Since there is a bound of overlapping sets $U_j$
when $j\in\mathbf{Z}_+$, we get
$$
\left ( \int _{U_j} 
|
\sum _{l=1}^\infty\fy _la
|^{G}_n(X)\, dX
\right )^p
\le
C_1  \left ( \sum _{k=0}^n \int _{U_j} 
|a|^{G}_k(X)\, dX
\right )^p,
$$
where the constant $C_1$ is independent of $j$. By
Lemma \ref{Lemma:EstIntermedDer} we obtain
\begin{align*}
\nm {\, |a|^{G}_n}{\WL^{1,p}_g}^p
&\le
C_1\sum _{j=1}^\infty \left ( \sum _{k=0}^n\int _{U_j} 
|a|^{G}_k(X)\, dX \right )^p|U_j|^{1-p}
\\[1ex]
&\le
C_2\sum _{j=1}^\infty\left ( \int _{U_j} 
\left (
|a(X)| + |a|^{G}_N(X)
\right )
\, dX \right )^p|U_j|^{1-p}
\\[1ex]
&\le
C_3\left (
\sum _{j=1}^\infty\left ( \int _{U_j} 
|a(X)|\, dX \right )^p|U_j|^{1-p}
\right .
\\
&{\phantom k} \qquad \qquad \qquad
\left .
+
\sum _{j=1}^\infty\left ( \int _{U_j} 
h_{g_j}^{\alpha N/2}(X)m(X)\, dX \right )^p|U_j|^{1-p}
\right )
\\[1ex]
&\asymp
\nm a{\WL^{1,p}_g}^p + \nm {h_gm}{\WL^{1,p}_g}^p\, ,
\end{align*}
for some constants $C_2$ and $C_3$.
This gives \eqref{m0}, and the proof is complete.
\end{proof}

\par

\par

\begin{rem}
By the proof of Lemma \ref{Lemma:EstIntermedDer},
it follows that the constant $C$ in
\eqref{Eq:DerWL1pEst} only depends on the dimension
of $W$ and on $N$.

\par

In particular, by changing the coordinates
in suitable ways, and using that there is
a bound of overlapping $U_j$, it follows that
\begin{equation}\label{Eq:DerWL1pEst2}
\Nm
{|a|_k^g}{\WL_{g,\theta} ^{q,p}}^p
\leq C
\left(
\nm a {\WL _{g,\theta}^{q,p}}^p
+
\Nm {|a|_N^g}{\WL _{g,\theta}^{q,p}}^p 
\right),
\quad k=0,1,\dots ,N.
\end{equation}
\end{rem}

\par

\begin{proof}[Proof of Theorem \ref{Thm:MainResult2}]
Let $G$ and $m_0$ be as in Lemma \ref{Lemma:m0uppsk}.
We observe that if
$a\in S(m,g)$, then $a\in S(m_0,G)$, in view of
\cite[Lemma 6.1]{Tof6}. The result now follows from
Theorem \ref{Thm:SymbClassEmbd}.
\end{proof}

\par

If the involved
weight functions are $g$-continuous,
we can replace the conditions on them
as in the next two theorems, where
the first one agrees with
\cite[Theorem 4.1]{Tof14} when $p\le 1$.

\par

\begin{thm}\label{Thm:SymbClassEmbdB}
Let $p\in (0,1]$, $g$ be feasible on $W$, and
$m\in L^p(W)$ be a positive $g$-continuous 
function on $W$.
Then $S(m,g)\subseteq s_p^w(W)$.
\end{thm}

\par

\begin{thm}\label{Thm:MainResult2B}
Let $p\in (0,1]$, $g$ be feasible on $W$,
$m$ be a positive $g$-continuous
function on $W$ such that
$h_g^{k/2}m\in L^p(W)$ for some $k\ge 0$,
and suppose $a\in S(m,g)\cap \WL _g^{1,p}(W)$. 
Then $a\in s_p^w(W)$.
\end{thm}

\par

Theorems \ref{Thm:SymbClassEmbdB}
and \ref{Thm:MainResult2B} are
straight-forward consequences
of Theorems \ref{Thm:SymbClassEmbd}
and \ref{Thm:MainResult2},
combined with the following lemma.
The details are left for the reader.

\par

\begin{lemma}\label{Lemma:WeightNormEquiv}
Let $p,q\in (0,\infty ]$, $g$ be slowly varying,
and $m$ be $g$-continuous on $W$. Then
$$
m\in L^p(W)
\quad \iff \quad
m\in \WL _{g}^{q,p}(W).
$$
\end{lemma}

\par

\begin{proof}
%
%
Suppose $m\in L^p(W)$, and
let $\{U_j\}_{j\in\mathbf{Z}_+}$ be an admissible $g$-covering of $W$ with
centers in $X_j\in W$, $j\in \mathbf Z_+$.
Since $m$ is $g$-continuous and $g$ is slowly varying, it follows that
$$
\nm m{L^p}^p \asymp
\sum _{j=1}^\infty m(X_j)^p|U_j| .
$$
By using the $g$-continuity again, it also
follows that
$$
\nm m{\WL _{g,\theta}^{q,p}}^p \asymp
\sum _{j=1}^\infty m(X_j)^p|U_j|^{\frac pq}
|U_j|^{\theta p}=
\sum _{j=1}^\infty m(X_j)^p|U_j|,
$$
and the asserted equivalence follows from
these relations.
\end{proof}

\par

\begin{rem}
Suppose that, in addition to the assumptions
of Theorem \ref{Thm:SymbClassEmbdB},
the metric $g$ and the weight
$m$ are $\sigma$-temperate
and $(\sigma ,g)$-temperate,
respectively. Then there is
a natural extension of
Theorem \ref{Thm:SymbClassEmbdB} to
Weyl operators acting on
Sobolev-type Hilbert spaces, $H(m,g)$,
introduced by Bony and Chemin in
\cite{BonChe}, which is especially
suitable for the Weyl-H{\"o}rmander
calculus.
(See also Section 2.6 in \cite{Ler}.)

\par

In fact, suppose that $m$ and $m_0$ are
$g$-continuous and $(\sigma ,g)$-temperate,
and $a\in S(m,g)$. Then 
$$
\op ^w(a):H(m_0,g)\to H(m_0/m,g)
$$
is continuous. In \cite{BonChe,Ler} it is
also shown that there are
$a_0\in S(m,g)$ and $b_0\in S(1/m,g)$ such that
\begin{equation}\label{Eq:WeylHormInverses}
\op ^w(b_0)= \op ^w(a_0)^{-1},
\qquad
a_0\in S(m,g),\ b_0\in S(1/m,g).
\end{equation}
Especially, it follows that
$$
\op ^w(a_0) : H(m_0,g) \to H(m_0/m,g)
\quad \text{and}\quad
\op ^w(b_0) : H(m_0/m,g) \to H(m_0,g)
$$
are continuous bijections, which are inverses
to each other. In particular, from these mapping 
properties it follows that equality is attained
in \eqref{Eq:CompWeylHormClasses}.

\par

Now let $p\in (0,1]$, $g$ be strongly feasible
on $W$, and $m$, $m_1$, and $m_2$ be
positive $g$-continuous and $(\sigma ,g)$-temperate
functions on $W$ such that
$$
\frac {m_2m}{m_1}\in L^p(W).
$$
A combination of Theorem \ref{Thm:SymbClassEmbdB}
and \eqref{Eq:WeylHormInverses} then gives
$$
S(m,g)\subseteq s_{A,p}(\maclH _1,\maclH _2),
\quad \text{when} \quad
\maclH _1 = H(m_1,g),\ \maclH_2 =H(m_2,g).
$$
(See also \cite[Theorem 4.4]{Tof14}.)
Since $H(1,g)=L^2(V)$, in view of
\cite{BonChe,Ler}, we regain 
Theorem \ref{Thm:SymbClassEmbdB}
in the case when $m$ is
$g$-continuous and $(\sigma ,g)$-temperate,
by choosing $m_1=m_2=1$.
\end{rem}

\par


\par

\subsection{Split metrics and more general
pseudo-differential calculi}\label{sec32}

\par

In order to state analogous results for
more general pseudo-differential calculi,
we need to impose further restrictions on
the metric $g$ and weight function $m$.

\par

We recall that a feasible metric $g$
on $W$ is called \emph{split}, if there are
global symplectic coordinates $Y=(y,\eta )$
such that
$$
g_X(y,-\eta ) = g_X(y,\eta ),
$$
for all $X\in W$.

\par

The next proposition follows from
\cite[Theorem 18.5.10]{Ho3} and its proof.
The details are left for the reader.

\par

\begin{prop}\label{Prop:CalculiInv}
Let $A,B\in \maclL(V)$, $g$ be strongly
feasible and split on $W=T^*V$, and let
$m$ be $g$-continuous and
$(\sigma ,g)$-temperate weight function.
Then
$$
\op _A(S(m,g)) = \op _B(S(m,g)).
$$
\end{prop}

\par

A combination of Theorem
\ref{Thm:SymbClassEmbdB}, Theorem
\ref{Thm:MainResult2B}, and Proposition
\ref{Prop:CalculiInv} gives the
following. The details are left for
the reader.

\par

\begin{thm}\label{Thm:SymbClassEmbdC2}
Let $A\in \maclL (V)$, $p\in (0,1]$,
$g$ be strongly feasible and split
on $W$, and $m\in L^p(W)$ be a positive
$g$-continuous and $(\sigma ,g)$-temperate
function on $W$.
Then $S(m,g)\subseteq s_{A,p}(W)$.
\end{thm}

\par

\begin{thm}\label{Thm:MainResult2C2}
Let $A\in \maclL (V)$, $p\in (0,1]$,
$g$ be strongly feasible and split on $W$,
$m$ be a positive $g$-continuous
and $(\sigma ,g)$-temperate function
on $W$ such that $h_g^{k/2}m\in L^p(W)$
for some $k\ge 0$.
Also, suppose $a\in S(m,g)\cap \WL _g^{1,p}(W)$. 
Then $a\in s_{A,p}(W)$.
\end{thm}

\par

\section{Applications to
special families of
pseudo-differential operators}\label{sec4}

\par

In this section we apply the results from
previous sections to obtain Schatten-von
Neumann properties for pseudo-differential
operators with symbols in the well-known
Shubin classes and SG classes 
(see \cite{Sh}).
We first recall their definitions.
Here, let
$$
\eabs x = (1+|x|^2)^{\frac 12},
\qquad
x\in \rr d.
$$

\par

\begin{defn}\label{Def:SpecSymbClasses}
Let $r,\rho \in \mathbf R$.
\begin{enumerate}
\item The \emph{Shubin class}
$\Sh ^r(\rr d)$ is the set of all
$f\in \mascC ^\infty (\rr d)$ such that
$$
|D^\alpha f(x)|
\le
C_\alpha \eabs x^{r-|\alpha |},
\qquad x\in \rr d.
$$

\vrum

\item The \emph{SG class}
$\SG ^{r,\rho}(\rr {2d})$ is the set of all
$a\in \mascC ^\infty (\rr {2d})$ such that
$$
|D_x^\alpha D_\xi ^\beta a(x,\xi )|
\le
C_{\alpha ,\beta} \eabs x^{r-|\alpha |}
\eabs \xi ^{\rho-|\beta |},
\quad
x,\xi \in \rr d.
$$
\end{enumerate}
\end{defn}

\par

\begin{rem}\label{Rem:SymbClassObs}
Let $p\in (0,1]$. For the symbol classes in 
Definition \ref{Def:SpecSymbClasses},
we observe the following:
\begin{enumerate}
\item
if $r\in\mathbf{R}$,
then $S(m,g) = \Sh ^r(\rr {2d})$ when
\begin{equation}\label{Eq:ShubMetric}
g_{x,\xi} (y,\eta) = \frac {|y|^2+|\eta |^2}
{\eabs{(x,\xi )}^2}
\quad \text{and}\quad
m(x,\xi )= \eabs{(x,\xi )}^r.
\end{equation}

Furthermore, $h_g(x,\xi )= \eabs{(x,\xi )}^{-2}$ and
$$
h_g^{k/2}m \in L^p(\rr {2d}),
\quad \text{when}\quad
k>r+\frac {2d}p\text ;
$$

\vrum

\item if $r,\rho \in\mathbf{R}$,
then $S(m,g) = \SG ^{r,\rho}(\rr {2d})$ when
\begin{equation}\label{Eq:SGMetric}
g_{x,\xi} (y,\eta) = \frac {|y|^2}
{\eabs x^2}+\frac {|\eta |^2}
{\eabs \xi ^2}
\quad \text{and}\quad
m(x,\xi )= \eabs {x}^r\eabs \xi ^{\rho}.
\end{equation}
Furthermore, $h_g(x,\xi )= (\eabs{x}\eabs \xi )^{-1}$
and
$$
h_g^{k/2}m \in L^p(\rr {2d}),
\quad \text{when}\quad
k>2\max(r,\rho )+\frac {2d}p.
$$
\end{enumerate}

In both (i) and (ii),
$g$ is strongly feasible and $m$ is
$g$-continuous and $(\sigma ,g)$-temperate.
\end{rem}

\par

In the next result we show how Lemma 
\ref{Lemma:WeightNormEquiv}
and Theorem \ref{Thm:MainResult2C2} can be combined 
with Remark \ref{Rem:SymbClassObs}, in order to 
obtain quasi-Banach Schatten-von Neumann properties 
for the Shubin classes and the SG classes.


\par

\begin{prop}
Let $p\in (0,1]$, $A$ be a real $d\times d$-matrix,
and $r,\rho \in \mathbf R$. Then the
following is true:
\begin{enumerate}
\item if $g$ is given by \eqref{Eq:ShubMetric},
then
$$
\Sh ^r(\rr {2d})\cap \WL ^{1,p}_g(\rr {2d})
\subseteq s_{A,p}(\rr {2d})\text ;
$$

\vrum

\item if $g$ is given by \eqref{Eq:SGMetric},
then
$$
\SG ^{r,\rho}(\rr {2d})\cap \WL ^{1,p}_g(\rr {2d})
\subseteq s_{A,p}(\rr {2d}).
$$
\end{enumerate}
\end{prop}

\par

\end{document}